\documentclass[11pt]{amsart}

\input{xy}
\xyoption{all}
\usepackage{graphicx}
\usepackage{color}
\usepackage{verbatim} 
\usepackage{mathrsfs}

\usepackage{hyperref}

\newtheorem{theorem}{Theorem} [section]
\newtheorem{prop}[theorem]{Proposition}
\newtheorem{lemma}[theorem]{Lemma}
\newtheorem{cor}[theorem]{Corollary}

\theoremstyle{definition}

\theoremstyle{remark}
\newtheorem{remark}[theorem]{Remark}

\numberwithin{equation}{section}
\numberwithin{figure}{section}

\newcommand\bP{{\mathbb P}}
\newcommand\C{{\mathbb C}}

\newcommand\N{{\mathbb N}}

\renewcommand\P{{\mathbb P}}

\newcommand\Q{{\mathbb Q}}

\newcommand\lra{\longrightarrow}

\renewcommand\phi{\varphi}

\renewcommand\O{\mathcal{O}}

\newcommand\Bif{\mathrm{Bif}}

\newcommand{\crit}{\mathrm{crit}}

\newcommand{\Lbar}{\overline{L}}

\newcommand\Gal{\operatorname{Gal}}

\newcommand\codim {\operatorname{codim}}
\newcommand\M {\mathrm{MP}}

\newcommand\Kbar {\overline{K}}
\newcommand\Qbar {\overline{\Q}}

\newcommand\ord {\operatorname{ord}}
\newcommand\hhat {\hat{h}}

\begin{document}

\title{The Dynamical Andr\'e-Oort Conjecture for cubic polynomials}

\author{Dragos Ghioca}
\address{
Dragos Ghioca\\
Department of Mathematics\\
University of British Columbia\\
Vancouver, BC V6T 1Z2\\
Canada
}
\email{dghioca@math.ubc.ca}

\author{Hexi Ye}
\address{
Hexi Ye\\
Department of Mathematics\\
University of British Columbia\\
Vancouver, BC V6T 1Z2\\
Canada}
\email{yehexi@math.ubc.ca}


\thanks{Our research was partially supported by an NSERC grant.}


\begin{abstract}
In the moduli space $\M_d$ of degree $d$ polynomials,  special subvarieties are those cut out by critical orbit relations, and then special points are the post-critically finite polynomials. It was conjectured that in $\M_d$, subvarieties containing a  Zariski-dense set of special points are exactly these special subvarieties. In this article, we prove the first non-trivial case for this conjecture: the case $d=3$. 
\end{abstract}

\maketitle

\section{Introduction}

Our main result is the proof of the first important case of the Dynamical Andr\'e-Oort Conjecture (posed by Baker and DeMarco \cite{Matt-Laura-2}); see Theorem \ref{precise version} for a more precise version, which also includes a Bogomolov-type statement for our result.
\begin{theorem}\label{DAO cubic} 
In the moduli space of cubic polynomials, the irreducible subvarieties containing a Zariski-dense set of post-critically finite points are exactly those cut out by critical orbit relations. 
\end{theorem}  

We describe next the background for both our result and for the Dynamical Andr\'e-Oort Conjecture. 
Given a variety $X$, a subvariety $V\subset X$, and given a family $\mathcal{Y}$ of subvarieties $Y\subseteq X$, the principle of \emph{unlikely intersections} predicts that $$V\cap \bigcup_{\substack{Y\in\mathcal{Y}\\ \dim(Y)<\codim(V)}}Y$$ is not Zariski dense in $V$, unless $V$ satisfies some rigid geometric property mirroring the varieties contained in $\mathcal{Y}$. Special cases of this principle of unlikely intersections in arithmetic geometry are the Bombieri-Masser-Zannier \cite{BMZ99}, the Pink-Zilber and the Andr\'e-Oort conjectures; for a comprehensive discussion, see the beautiful book of Zannier \cite{Zannier-book}. Motivated by a version of the Pink-Zilber Conjecture for semiabelian schemes, Masser and Zannier (see \cite{M-Z-1, M-Z-2}) proved that in a non-constant elliptic family $E_t$ parametrized by $t\in \C$, for any two sections $\{P_t\}_t$ and $\{Q_t\}_t$, if there exist infinitely many $t\in \C$ such that both $P_t$ and $Q_t$ are torsion points on $E_t$, then the two sections are linearly dependent. 

The results of Masser-Zannier \cite{M-Z-1, M-Z-2} have natural dynamical reformulations using the Latt\'es maps associated to the elliptic curves in the family $\{E_t\}$. More generally, one can consider the following problem: given a curve $C$ defined over $\C$, given a family of rational functions $f_t\in \C(z)$  parametrized by the points $t\in C(\C)$, and given $a,b:C\lra \bP^1$, then one expects that there exist infinitely many $t\in C(\C)$ such that both $a(t)$ and $b(t)$ are preperiodic under the action of $f_t$ if and only if $a$ and $b$ are dynamically related with respect to the family $\{f_t\}$ in a precise manner.  The first result in this direction, for a family $f_t$ which was not induced by the endomorphism of an algebraic group was proven by Baker and DeMarco \cite{Baker-DeMarco}. They answered a question of Zannier, thus showing that for an integer $d\ge 2$, and for two complex numbers $a$ and $b$, if there exist infinitely many $t\in\C$ such that both $a$ and $b$ are preperiodic under the action of $z\mapsto z^d+t$, then $a^d=b^d$. Several new results followed (see \cite{GHT-ANT, Matt-Laura-2, GHT:preprint, AO1, AO2}), but still the question stated above remains open in its full generality; especially, it is particularly difficult to treat the case when $C$ is an arbitrary curve, even when dealing with families of polynomials. For example, in  \cite{GHT:preprint} (which is one of the very few articles treating the case when $C\ne \bP^1$), the family of rational functions must have exactly one degenerate point on $C$ and also the family $\{f_t\}$  must satisfy additional technical conditions. In \cite{AO2}, the parameter curve is arbitrary, but the result holds only for families of unicritical polynomials, which is again quite restrictive.  In this article, we release all the restrictions on the curve $C$, which parametrizes a family of  cubic polynomials. The dynamics of cubic polynomials is already significantly more involved than the dynamics of unicritical polynomials, and therefore our analysis is significantly more involved than in the article \cite{AO2}.

In \cite{Matt-Laura-2}, Baker and DeMarco posed a very general question for families of dynamical systems, which is motivated by the classical Andr\'e-Oort conjecture. As a dynamical analogue to the classical Andr\'e-Oort Conjecture, Baker and DeMarco's question asks that if a subvariety $V$ of the moduli space of rational maps of given degree contains a Zariski dense set of post-critically finite points, then $V$ itself is  cut out by \emph{critical orbit relations}, i.e., the critical points of the rational functions in the family $V$ are related dynamically (see condition~(2) in Theorem~\ref{precise version}; for more details,  see \cite{Matt-Laura-2}). We recall that a rational function $f$ is \emph{post-critically finite} (PCF) if each critical point of $f$ is preperiodic. Also, we recall the classical notation and definition from algebraic dynamics that for a rational function $f$, its $n$-th iterate is denoted by $f^n$, and that a point $c$ is \emph{preperiodic} if and only if there exist integers $0\le m<n$ such that $f^m(c)=f^n(c)$.

In this article we prove the first case of the Dynamical Andr\'e-Oort Conjecture (posed by Baker and DeMarco \cite{Matt-Laura-2}), solving completely their conjecture in the moduli space $\M_3$ of cubic polynomials. Furthermore, we  obtain a \emph{Bogomolov-type} statement for our result, i.e., the condition that a curve $C\subset \M_3$ contains infinitely many post-critically finite points can be relaxed to asking that $C$ contains infinitely many points of small critical height. We recall that for a polynomial $f\in \Qbar[z]$ of degree $d\ge 2$, the \emph{canonical height} of a point $c\in\Qbar$ is equal to 
\begin{equation}
\label{first canonical height}
\hhat_f(c):=\lim_{n\to\infty}\frac{h(f^n(c))}{d^n},
\end{equation} 
where $h(\cdot )$ is the usual Weil height; for more details on the canonical height we refer the reader to the article of Call and Silverman \cite{Call:Silverman}. Then for a polynomial $f\in \Qbar[z]$ of degree $d\ge 2$, we define the \emph{critical height} of a polynomial be 
\begin{equation}
\label{first critical height}
\hhat_{\crit}(f):=\sum_{i=1}^{d-1}\hhat_f(c_i),
\end{equation}
where the $c_i$'s are the critical points of $f$ (other than $\infty$). Clearly, $f\in \Qbar[z]$ is post-critically finite if and only if $\hhat_{\crit}(f)=0$.

\begin{theorem}\label{precise version}
Let $C\subset \C^2$ be an irreducible curve and let $f_{a,b}(z)=z^3-3a^2z+b$; then the two critical points of $f_{a,b}(z)$ are $\pm a$. The following statements are equivalent: 
\begin{itemize}
\item[(1)] there are infinitely many $(a,b)\in C$ with $f_{a,b}$ being post-critically finite. 
\item[(2)] $C$ is the line in $\C^2$ given by the equation $b=0$, or $C$ is an irreducible component of a curve defined by one of the following equations:
\begin{itemize}
\item[(i)] $f_{a,b}^n(a)=f^m_{a,b}(a)$ for some $n>m\geq 0$; or
\item[(ii)] $f_{a,b}^n(-a)=f^m_{a,b}(-a)$ for some $n>m\geq 0$; or 
\item[(iii)] $f^n_{a,b}(a)=f^m_{a,b}(-a)$ for some $n, m\geq 0$.
\end{itemize}
\item[(3)] $C$ is defined over $\Qbar$, and there is sequence of non-repeating points $(a_n, b_n)\in C(\Qbar)$ with $\lim_{n\to \infty}\hhat_{\crit}(f_{a_n, b_n})=0$. 
\end{itemize}
\end{theorem}

For the case $b=0$, the two critical points $\pm a$ and also their orbits are related by $\sigma(z)=-z$, where $\sigma$ is the unique symmetry of the Julia set of $f_{a,0}$ with $f_{a,0}\circ \sigma=\sigma\circ f_{a,0}$.  We observe that condition~(1) easily yields condition~(3) in Theorem~\ref{precise version} since each post-critical point in $\M_3$ corresponds to a polynomial with algebraic coefficients, and furthermore (as shown in \cite{Call:Silverman}), a point is preperiodic if and only if its canonical height equals $0$. Also, a curve of the form~(2) as in the conclusion of Theorem~\ref{precise version} contains infinitely many PCF points. Therefore, the difficulty in Theorem~\ref{precise version} lies with proving that it is \emph{only} curves cut out by orbit relations which contain an infinite sequence of points of critical height converging to $0$.

A key ingredient of  our article (and also of all of the articles previously mentioned dealing with the problem of unlikely intersection in algebraic dynamics) is the arithmetic equidistribution of small points on an algebraic variety (in the case of $\P^1$, see \cite{Baker-Rumely06, favre-rivera06}, in the general case of  curves, see \cite{CLoir, Thuillier}, while for arbitrary varieties, see \cite{Yuan}). One of the significant technical difficulties of this article is to show the equidistribution of preperiodic parameters for a marked critical point of cubic polynomials on an arbitrary curve $C$, \emph{not necessarily containing infinitely PCF points}.  Precisely, it is delicate to prove the continuity of potential functions of the bifurcation measures at infinity points, especially when the bifurcation locus is not compact. Moreover, because the sets $S^\pm\backslash S_0^\pm$ in (\ref{pole set}) can be nonempty, \emph{unlike} all of the previous articles on this theme, we cannot obtain the  desired metrized line bundle from a uniform limit of semi-positive metrics on the same line bundle. See Remark \ref{remark 11} for the following result.
\begin{theorem}
Let $C$ be an irreducible plane curve defined over a number field $K$, such that $\pm a$ is not persistently preperiodic for $f_{a,b}$ on $C$. The set of preperiodic parameters 
   $$\textup{Preper}_\pm:=\{t\in C: ~ \pm a(t) \textup{ is preperiodic for } f_{a(t), b(t)}\}$$
is equidistributed with respect to bifurcation measure $\mu_\pm$. More precisely, for any sequence of non-repeating points $t_n\in \textup{Preper}_+$, the discrete probability measures 
    $$\mu_n=\frac{1}{|\Gal(\Kbar/K) \cdot t_n|}\sum_{t\in \Gal(\Kbar/K) \cdot t_n}\delta_t$$
converge weakly to the normalized measure $\mu_+/\mu_+(C)$; similarly for $\textup{Preper}_-$ and $\mu_-/\mu_-(C)$.  
\end{theorem}

\noindent{\bf Outline of the article}. Using the arithmetic equidistribution theorem \cite{CLoir, Yuan} we show that preperiodic parameters for a marked critical point equidistribute on the curve with respect to the bifurcation measure; see Theorem \ref{equidistribution of parameters}. Assuming there exist infinitely many points $(a,b)$ on the plane curve $C$ such that $z^3 - 3a^2z+b$ is PCF, then the potential (escape-rate) functions for the bifurcation measures (with respect to the starting points $\pm a$) are proportional to each other; see Corollary~\ref{proportional escape cor}. Finally, using the classification of Medvedev-Scanlon \cite{M-S-1} for invariant plane curves under the coordinatewise action of polynomials, we derive a polynomial relation between the critical points of the cubic polynomial; see Theorem~\ref{algebraic relation}.

\medskip

{\bf Acknowledgments.} We are indebted to Holly Krieger and to Khoa Nguyen for numerous stimulating discussions which helped us considerably while preparing this manuscript.  We also thank Laura DeMarco,  Xiaoguang Wang and Xinyi Yuan for many useful comments and suggestions which improved our presentation. At the time we were finishing writing this article, we learned that Charles Favre and Thomas Gauthier have independently obtained a proof of Theorem \ref{DAO cubic} using a different approach; we are grateful to both of them for sharing with us their preprint.

\section{Dynamics of cubic polynomials}
\label{section dynamics cubics}

In this section, we first introduce the moduli space $\M_3$ of cubic polynomials, and then study the dynamics of the two critical points of the cubic polynomials. 

\subsection{Moduli space of cubic polynomials} Let $f:\C \to \C$ be a polynomial of degree $d\geq 2$. Two polynomials $f$ and $g$ are conjugate to each other if there is an affine map $A(z)=az+b$ with $f=A^{-1}\circ g\circ A$. The {\em moduli space}  $\M_d$ of polynomials consists of conjugacy classes $[f]$ of polynomials of $d\geq 2$, and $\M_d$ is an affine space of dimension $d-1$. Moreover, it is well known that PCF points in $\M_d$ are countable and Zariski-dense; in particular, each PCF polynomial has coefficients in $\Qbar$. 

In this article, we focus mainly on the moduli space $\M_3\simeq\C^2$ of degree three polynomials. 
Each cubic polynomial can be conjugate to a monic and \emph{centred} polynomial, i.e. 
  $$f_{a, b}(z):=z^3-3a^2z+b$$
with two critical points $\pm a$ (counted with multiplicity). Two polynomials $f_{a_1, b_1}, f_{a_2, b_2}$ of such form are in the same conjugacy class if and only if  $a_1=\pm a_2$ and $b_1=\pm b_2$. Hence the map from $\M_3^{cm}\simeq \C^2$ to $\M_3$, given by $f_{a,b}\to [f_{a,b}]$ is a $4$-to-$1$ map with ramification when $a=0$ or $b=0$. The two critical points of polynomials $f_{a,b}\in \M_3^{cm}$ have been marked as $\pm a$, so we call $\M_3^{cm}$ the moduli space of cubic polynomials with marked critical points.  As $+a$ and $-a$ share lots of properties under the iteration of $f_{a,b}$, we focus on the study of the iterates of $+a$. 

\subsection{Order of poles} Let $C$ be an irreducible curve in $\C^2$. By abuse of notation, we also denote by $C$ the normalization of its Zariski closure in $\P^2$; hence we view $C$ also as a smooth projective curve. Let $c_n:=f^n_{a,b}(a)$, which is a rational map $C\lra \bP^1$; also, $c_n$ can be viewed as a polynomial in the two rational maps $a,b:C\lra \bP^1$. Finally, we note that \emph{not both} $a$ and $b$ are constant on $C$.

\begin{lemma}\label{bounded degree} On $C$, the marked critical point $+a$ is persistently preperiodic under $f_{a,b}$, if and only if the degree of $c_n\in \C(C)$ is uniformly bounded for $n\geq 1$. 
\end{lemma} 
\proof See \cite{Baker} and also \cite{DeMarco: heights}; note that the family of polynomials $z^3-3a^2z+b$ is not isotrivial since noth both $a$ and $b$ are constant functions in $\C(C)$. \qed

Let $t_0\in C(\C)$ which is a pole of either $a$ or $b$ on the smooth projective curve $C$. Note that since not both $a$ and $b$ are constant on $C$, then there must exist such a pole $t_0\in C$. Let  $\gamma_a:=-\ord_{t_0} a$,  $\gamma_b:=-\ord_{t_0} b$ and
\begin{equation}\label{gamma n}
 \gamma_{\max}:=\max\{3\gamma_a, \gamma_b\}>0, ~\textup{and } \gamma_n:=-\ord_{t_0} c_n.
 \end{equation}
 
\begin{lemma}\label{order of pole growth}
For a point $t_0\in C$ with $\gamma_{\max}>0$, we have
\begin{itemize}
\item if $\gamma_{\max}<3 \gamma_{n_0}$ for some $n_0\geq 1$, then for all $n\geq n_0$, $\gamma_n=3^{n-n_0}\cdot \gamma_{n_0}$. In particular, if $\lim_{t\to t_0}2a^3(t)/b(t)\neq 1$, then $\gamma_n=3^{n-1}\cdot \gamma_{\max}$. 
\item if $\lim_{n\to \infty} \gamma_n\neq \infty$, then $\gamma_b=3\gamma_a$ and $\gamma_n=\gamma_a$ for all $n\geq 1$. 
\end{itemize}
\end{lemma}
\proof From the fact that 
\begin{equation}\label{induction formula}
c_{n+1}=f_{a,b}(c_n)=c_n^3-3a^2c_n+b,
\end{equation}
it is clear that if $3\gamma_n>\gamma_{\max}$ then $\gamma_{n+1}=3\gamma_n$. Consequently, if $3\gamma_{n_0}>\gamma_{\max}$, then $\gamma_n=3^{n-n_0}\cdot \gamma_{n_0}$. When $\lim_{t\to t_0}2a^3(t)/b(t)\neq 1$, $\gamma_1=-\ord_{t_0} (b-2a^3)=\gamma_{\max}$, hence $\gamma_n=3^{n-1}\cdot \gamma_{\max}$ for all $n\geq 1$. 

Now suppose that $\lim_{n\to \infty} \gamma_n\neq \infty$. From the above analysis we get $\lim_{t\to t_0}2a^3(t)/b(t)=1$ and then $\gamma_b=3\gamma_a=\gamma_{\max}$. Moreover, we have $\gamma_n\leq \gamma_a$ for all $n$. If there is some $n_1$ with $\gamma_{n_1}<\gamma_a$, then from equation (\ref{induction formula}), $\gamma_{n_1+1}=\gamma_{\max}$. Let $n_0=n_1+1$, then $3\gamma_{n_0}>\gamma_{\max}$ and consequently $\gamma_n=3^{n-n_0}\cdot \gamma_{\max}$ would tend to infinity as $n\to \infty$, which is a contradiction.
\qed

From the behaviour of the order of poles in Lemma \ref{order of pole growth}, we define the following sets in the smooth projective curve $C$
\begin{equation}\label{pole set}
S^+:=\{t_0\in C| \gamma_{max}>0\}\textup{ and } S_0^+:=\{t_0\in C| \lim_{n\to \infty} \gamma_n=\infty\}
\end{equation}
It is clear that $S_0^+$ is subset of $S^+$, and similarly we can define $S_0^-\subset S^-\subset C$ for the marked critical point $-a$. Since all the poles of $c_n\in \C(C)$ are in $S^+$, then Lemma \ref{bounded degree} yields the following result.
\begin{lemma}\label{non empty poles}
The marked critical point $+a$ (resp., $-a$) is persistently preperiodic under $f_{a, b}$ on $C$, if and only if $S_0^+$ (resp., $S_0^-$) is empty.  
\end{lemma}


\subsection{Escape-rate function and bifurcation} For the marked critical points $\pm a$, the {\em escape-rate functions} $G^\pm$ are given by 
 \begin{equation}\label{escape-rate definition}
 G^\pm (a, b):=\lim_{n\to \infty}\frac{1}{3^n}\log^+|f^n_{a,b}(\pm a)|,
 \end{equation}
for $(a,b)\in \C^2$, where $\log^+ |z|:=\max\{\log |z|, 0\}$. In the above formula, the convergence is local and uniform on $\C^2$; see the proof of Lemma \ref{convergence 1}. Then the escape-rate functions $G^+(a,b)$ are  continuous and plurisubharmonic on $\C^2$, and then they are subharmonic when restricted on an irreducible curve $C\subset \C^2$. The escape-rate function $G^{\pm}(a, b)\geq 0$ with equality if and only if the iterates $f_{a,b}^n(\pm a)$ of the  critical point $\pm a$ do not tend to infinity as $n$ tends to infinity. 

The {\em bifurcation measures} $\mu_{\pm}$ on $C$ corresponding to marked critical points $\pm a$ are given by 
\begin{equation}\label{bifurcation measures}
\mu_\pm:=dd^c G^\pm.
\end{equation}
The {\em bifurcation locus} $\Bif_\pm$ for marked critical points $\pm a$  is the set of parameters on $C\subset \C^2$,  such that 
  $$\{f^n_{a(t), b(t)}(\pm a(t))\}_{n\geq 1}$$
 is not a normal family on any small neighbourhood of such parameters. Actually, $\Bif_\pm\subset C$ is the boundary of the parameters $t\in C$ with $G^\pm(a(t), b(t))=0$, and then from the continuity of the escape-rate functions it has $G^\pm=0$ on $\Bif_\pm$. The supports of the bifurcation measures $\mu_\pm$ on $C$ are exactly $\Bif_\pm$. The bifurcation locus $\Bif$ on $C$ is the union of $\Bif_+$ and $\Bif_-$, and it is the set of parameters where the dynamical system $f_{a(t), b(t)}(z)$ is unstable when we perturb the parameter $t$. 

\begin{prop}\label{infinite preperiodic points}
For any irreducible curve $C\subset \C^2$, there are infinitely many points $(a,b)\in C$ such that $+a$ (or $-a$) is preperiodic under the iteration of $f_{a,b}$. 
\end{prop}
\proof This proposition follows from Montel's Theorem; see \cite[Theorem 5.1]{DeMarco: heights}.\qed

We thank Xiaoguang Wang for suggesting the proof of the following proposition. 
\begin{prop}\label{proportional escape-rate}
Let $C$ be an irreducible curve $C\subset \C^2$. Suppose  $\mu^+=r\mu^-$ on $C$ for some $r>0$, then we have $G^+=rG^-$ on $C$. 
\end{prop}
\proof As $\Bif_\pm$ are the supports of the $\mu_\pm$, then $\Bif=\Bif_+=\Bif_-$. Since $\mu^+=r\mu^-$, $G^+-rG^-$ is harmonic on $C$ and  zero on $\Bif$. From \cite{McMullen:Universal}, in any small neighbourhood of  a point in $\Bif$, we can find a small (topological) generalized Mandelbrot set with boundary in $\Bif$,  then $G^+-rG^-$ is zero on the boundary of this small generalized Mandelbrot set.  Hence the harmonic function $G^+-rG^-$ is zero in the interior of this generalized Mandelbrot set. Consequently, $G^+-rG^-$ is zero everywhere on $C$ as it is harmonic. \qed


\section{A line bundle with continuous metric}\label{continuous line bundle}In this section, we construct  metrized line bundles on an irreducible curve $C\subset \M_3^{cm}$, associated to a critical point which is not persistently preperiodic on $C$. One of the main goals of this section is to show that the metric on the line bundle is continuous, which is crucial  in the next section for proving the equidistribution of small points. 

\subsection{Continuity of the escape-rate function} We show continuity of the escape-rate function $G^+$ restricted on the smooth projective curve $C$. A {\em uniformizer} at $t_0\in C$ is a rational function $u\in \C(C)$ such that $\ord_{t_0} u=1$. 
\begin{theorem}\label{continuity thm}
For the escape-rate function $G^+$ restricted on a smooth projective curve $C\backslash S^+$, we have 
\begin{itemize}
\item $G^+$ can be extended to a continuous subharmonic function on $C\backslash (S^+\backslash S_0^+)$, and $G^+=0$ at any $t_0\in S^+\backslash S_0^+$. 
\item For any $t_0\in S_0^+$ and $u$ being a uniformizer at $t_0$, the function $G^++\frac{\gamma_n}{3^n} \log |u|$ can be extended to a harmonic function in a neighbourhood of $t_0$ for any sufficiently large $n$, where $\gamma_n$ is defined in (\ref{gamma n}). 
\end{itemize}
\end{theorem}
The proof of this theorem follows easily from Lemmas \ref{convergence 1}, \ref{convergence 2} and \ref{continuity at bad point}. First, we mention the following convention for our forthcoming analysis. Given $t\in C(\C)$, for the sake of simplifying our notation, we will often drop the dependence on $t$ and simply use $a,b,c_n$ instead of $a(t), b(t), c_n(t)$.

\begin{lemma}\label{convergence 1}
The sequence of functions $\frac{1}{3^n}\log^+|f^n_{a,b}(+a)|$ converges locally uniformly to $G^+(a,b)$ on $C\backslash S^+$ as $n$ tends to infinity.  
\end{lemma}
\proof Let $M$ be a large positive real number, and assume that $3|a|^2, |b|<<M$. Recall that $c_n=f^n_{a,b}(+a)$ and $c_{n+1}=c_n^3-3a^2c_n+b$. Then we get that if $|c_n|<M$
  $$\left|\frac{\log^+|c_{n+1}|}{3^{n+1}}-\frac{\log^+|c_{n}|}{3^{n}}\right|<\frac{\log 3M^3}{3^{n+1}}$$
and if $|c_n|\geq M$
  $$\left|\frac{\log^+|c_{n+1}|}{3^{n+1}}-\frac{\log^+|c_{n}|}{3^{n}}\right| =\left|\frac{\log|1-3a^2/c_n^2+b/c_n^3|}{3^{n+1}}\right|<\frac{\log 2}{3^{n+1}}.$$
Hence $\frac{1}{3^n}\log^+|f^n_{a,b}(+a)|$ is a Cauchy sequence which converges locally uniformly. \qed

\begin{lemma}\label{convergence 2}
If $t_0\in S_0^+$ and if $u$ is a  uniformizer for $t_0$, then
  $$\frac{1}{3^n}\log^+|f^n_{a,b}(+a)|+\frac{\gamma_n}{3^n}\log|u| =\frac{1}{3^n}\log|u^{\gamma_n}f^n_{a,b}(+ a)|$$ converges locally uniformly  on $C$ at $t_0$ as $n$ tends to infinity.  
\end{lemma}
\proof As $\gamma_n$ is the order of pole of $c_n$ at $t_0$, $u^{\gamma_n}c_n$ is holomorphic and non-vanishing near $t_0\in C$, i.e., $\frac{1}{3^n}\log|u^{\gamma_n}f^n_{a,b}(+ a)|$ is harmonic near $t_0$ on $C$.  From Lemma \ref{order of pole growth}, we can pick a large $n_0$, such that $\gamma_{n+1}=3\gamma_n>>\gamma_{\max}$ for any $n\geq n_0$. As $c_{n+1}=c_n^3-3a^2+b$, inductively, $|c_n(t)|>>1$ grows exponentially fast as $n\to \infty$ for $t\in C$ very close to $t_0$. So for $t$ near $t_0$ and $n\geq n_0$, we have
$$\frac{1}{3^n}\log^+|c_n|+\frac{\gamma_n}{3^n}\log|u| =\frac{1}{3^n}\log|u^{\gamma_n}c_n|$$
 and $|3a^2(t)/c_n^2(t)|<<1, |b(t)/c_n^3(t)|<<1$, hence
   $$\left|\frac{\log|u^{\gamma_{n+1}}c_{n+1}|}{3^{n+1}}-\frac{\log|u^{\gamma_n}c_{n}|}{3^{n}}\right| =\left|\frac{\log|1-3a^2/c_n^2+b/c_n^3|}{3^{n+1}}\right|<
   \frac{\log 2}{3^n},$$
as desired.
\qed 
   
It is more delicate to show the continuity of the escape-rate function $G^{\pm}$ on $C$ crossing $t_0\in S^+\backslash S_0^+\subset C$. We first prove a result which will be used later in Lemma \ref{continuity at bad point}. 

\begin{lemma}\label{finitely many coefficients}
Let $\gamma$ be a positive integer, and  let $\tilde b(t)=2+\beta_1t+\beta_2 t^2+\beta_3 t^3+\cdots$ be a holomorphic germ at $t=0$. Define
 $$g(z):=z^3-3z+\tilde b.$$
Suppose $z_n(t)$ is a sequence of holomorphic germs at $t=0$, satisfying 
 \begin{itemize}
 \item $z_{n+1}(t)=g(z_n(t))/t^{2\gamma}$ for $n\geq 1$. 
 \item $\ord_{t=0} z_{n}(t)=0$ for all $n\geq 1$. 
 \end{itemize}
We write $z_n(t)=\sum^{\infty}_{i=0} x_{n,i} t^i$ for each $n\ge 1$. 
 Then there is a finite set $T$ such that $x_{n,i} \in T$ for all $n\geq 1$ and all $0\leq i\leq \gamma$; moreover, $x_{n,0}$ is a solution of the equation $x^3-3x+2=0$.
 \end{lemma}
 \proof Let 
    $$g(z_n(t))=z_n^3-3z_n+\tilde b=:\sum_{i=0}^{\infty}\alpha_{n,i} t^i.$$
Since $g(z_n)=z_{n+1}\cdot t^{2\gamma}$ and $\ord_{t=0} z_{n+1}(t)=0$, we have 
  $$\alpha_i=0, ~\textup{ for }0\leq i \leq 2\gamma\text{ and }\alpha_{n,2\gamma}\ne 0.$$
 Notice that since $\alpha_0=x_{n,0}^3-3x_{n,0}+2=0$, then 
$x_{n,0}\in \{1,-2\}$. Now suppose $x_{n,i}$ is unique up to finitely many choices for all $0\leq i\leq j-1<\gamma$, which is true when $j=1$; we prove next that also $x_{n,j}$ is unique up to finitely many choices. 

Suppose $x_{n,0}\neq 1$, i.e., $x_{n,0}=-2$. An easy computation shows that 
   $$\alpha_{n,j}=x_{n,j}(3x_{n,0}^2-3)+\beta_j+F_j(x_{n,0},\dots,x_{n,j-1})=0$$
where $F_j$ is a unique polynomial in $j$ variables obtained from the expansion of $z_n(t)^3$. Because $x_{n,0}= -2$, then $x_{n,j}$ is uniquely determined by $x_{n,0}, \dots, x_{n,j-1}$ and $\beta_j$. 

Now suppose $x_{n,0}=1$. Now, if $x_{n,i}=0$ for $1\leq i\leq j-1$, we have  that 
   $$\alpha_{2j}=3x_{n,j}^2+\beta_{2j}.$$
If $j<\gamma$, then $\alpha_{2j}=3x_{n,j}^2+\beta_{2j}=0$. Otherwise if $j=\gamma$, $\alpha_{2j}=3x_{n,j}^2+\beta_{2j}$ is the constant term of $z_{n+1}(t)$ which must satisfy the equation $x^3-3x+2=0$. In any case, $x_{n,j}$ is unique up to finitely many choices. 

Now, if $x_{n,i}\neq 0$ for some $1\leq i\leq j-1$ (also under the assumption that $x_{n,0}=1$), we let $\ell\ge 1$ be the smallest such integer with $x_{n,\ell}\neq 0$. Then 
   $$\alpha_{\ell+j}=6x_{n,\ell} x_{n,j}+G_{j,\ell}(x_{n,\ell},\dots, x_{n,j-1})+\beta_{\ell+j}=0$$
 where $G_{j,\ell}$ is a unique polynomial in $j-\ell$ variables obtained from the expansion of $z_n(t)^3$ (also taking into account that $x_{n,0}=1$ and that $x_{n,i}=0$ for $1\le i\le \ell-1$). Hence $x_{n,j}$ is also uniquely determined up to finitely many choices. By induction, for $0\leq i\leq \gamma$, we get that $x_{n,i}$ is uniquely determined by $\tilde b$ up to finitely many choices. \qed

\begin{lemma}\label{continuity at bad point}
Fix a point $t_0\in S^+\backslash S_0^+$ and a uniformizer $u$ for $t_0$ on $C$. The sequence of functions $\frac{1}{3^n}\log^+|u^{\gamma_n}f^n_{a,b}(+ a)|$ converges locally uniformly  on $C$ at $t_0$ as $n$ tends to infinity.  
\end{lemma}
\proof From Lemma \ref{pole set}, $\gamma_n=\gamma_a>0$ for all $n$, then the convergence statement is independent on the choice of uniformizer $u$. For a suitable choice of the uniformizer and also for a suitable analytic parameterization of a neighborhood of $t_0\in C$, we can assume that  
   $$t_0=0\text{, }u(t)=t\text{, }a(t)=t^{-\gamma_a}\text{, }b(t)=\beta_0t^{-3\gamma_a}+\beta_1t^{-3\gamma_a+1}+\beta_2t^{-3\gamma_a+2}+\cdots$$
and then $c_n(t)=f^n_{a(t), b(t)}(a(t))$. From Lemma \ref{pole set}, $\ord_{t=0}c_n(t)=-\gamma_a$ for all $n$ and 
  $$\lim_{t\to 0} 2a^3(t)/b(t)=1=2/\beta_0\text{ and thus $\beta_0=2$.}$$
The statement of this Lemma is equivalent to the statement that the sequence 
  $$\frac{1}{3^n}\log^+|t^{\gamma_a}c_n(t)|$$
converges locally uniformly in a neighbourhood of $t=0$ as $n\to \infty$. Also, by Lemma \ref{convergence 1}, $\frac{1}{3^n}\log^+|t^{\gamma_a}c_n(t)|$ converges locally uniformly for $t\neq 0$. As $ \log^+|t^{\gamma_a}c_i(t)|$ is always non-negative, to prove this lemma, it suffices to show that there exists a sequence of positive real numbers $r_n$ shrinking to zero as $n\to \infty$, such that
\begin{equation}\label{local limit}
\lim_{n\to\infty} \sup_{i\geq n} \sup_{|t|<r_n} \frac{ \log^+|t^{\gamma_a}c_i(t)|}{3^i}=0.
\end{equation}
   
Since $\ord_{t=0}c_n(t)=-\gamma_a$, the Taylor series of $c_n(t)$ can be written as
$$c_n(t)=x_{n,0}t^{-\gamma_a}+x_{n,1}t^{-\gamma_a+1}+x_{n,2}t^{-\gamma_a+2}\cdots.$$
For simplicity, we let 
  $$\tilde{c}_n(t):=t^{\gamma_a}c_n(t)=x_{n,0}+x_{n,1}t+x_{n,2}t^2+\cdots.$$
As  $c_{n+1}=c_n^3-3a^2c_n+b$,  we have that $\tilde c_{n+1}$ and $\tilde c_n$ satisfy
  $$\tilde c_{n+1}(t)=g(\tilde c_n)/t^{2\gamma_a}$$
for  $g(z):=z^3-3 z+\tilde b$ with $$\tilde b(t):=t^{3\gamma_a}b(t)=2+\beta_1t+\beta_2t^2+\cdots.$$
Let $r_0$ be a positive real number less than $1$ such that if $|t|\le r_0$ we have
\begin{equation}
\label{r_0 and tilde b}
|\tilde{b}(t)|<3.
\end{equation}

From Lemma~\ref{finitely many coefficients}, we know that there exists a finite set $S$ such that $x_{n,i}\in S$ for each $n\geq 1$ and each  $0\leq i\leq \gamma_a$; moreover, one has $|x_{n,0}|<3$. So, at the expense of replacing $r_0$ by a smaller positive real number, we may also assume that if $|t|\le r_0$, then
\begin{equation}
\label{sum x_i}
\left|\sum_{i=0}^{\gamma_a} x_{n,i}t^{i}\right| < 4.
\end{equation}

Let $M_0$ be a positive real number with the property that if $A>M_0$, then 
\begin{equation}
\label{simple inequality}
(A-10)^3 + 3(A-10) < A^3 - 30.
\end{equation}

Let $M\ge M_0$ be a real number, let $n_0$ be a large positive integer, and let $r_{n_0}< r_0$ be a positive real number such that for all $t$ with $|t|\leq r_{n_0}$, we have
 \begin{equation}\label{inductive 111}
 |\tilde c_{n_0}(t)|<M-10.
 \end{equation}
Inequalities \eqref{simple inequality} and \eqref{r_0 and tilde b} yield that if $|t|\le r_{n_0}$ then   
 \begin{equation}\label{inductive 112}
|g(\tilde c_{n_0}(t))|=|\tilde c^3_{n_0}-3\tilde c_{n_0}+\tilde b| < (M-10)^3+3(M-10)+ 3< M^3-27<M^3-10.
 \end{equation}  
Hence for $|t|\leq r_{n_0}$, by the Maximal Value Theorem, we have that $$|\tilde c_{n_0+1}(t)|=\left|\frac{g(\tilde{c}_{n_0}(t))}{t^{2\gamma_a}}\right|\leq \frac{M^3-10}{r_{n_0}^{2\gamma_a}},$$ and then (again using \eqref{simple inequality} and that $r_{n_0}<r_0<1$) we get
 \begin{equation}\label{inductive 113}
|\tilde c_{n_0+1}(t)|<\frac{M^3}{r_{n_0}^{2\gamma_a}}-10\text{ and so, } |g(\tilde c_{n_0+1})|<\left(\frac{M^3}{r_{n_0}^{2\gamma_a}}\right)^3-10.
\end{equation}
From (\ref{inductive 111}, \ref{inductive 112}, \ref{inductive 113}), inductively, we know that for $|t|\leq r_{n_0}$ and all $i\geq 0$, one has
 \begin{equation}\label{inductive 211}
 |\tilde c_{n_0+i}(t)|<\frac{M^{3^i}}{r_{n_0}^{2\gamma_a(3^{i-1}+3^{i-2}+\cdots+3^0)}}-10<\frac{M^{3^i}}{r_{n_0}^{2\gamma_a(3^{i-1}+3^{i-2}+\cdots+3^0)}}.
 \end{equation}
 So we have 
  $$\sup_{i\geq n_0}\sup_{|t|\leq r_{n_0}}\frac{\log^+|\tilde c_i(t)|}{3^i}\leq \frac{\log M}{3^{n_0}}-\frac{\log r_{n_0}^{2\gamma_a}}{3^{n_0}}.$$
  
Next, we are going to shrink $r_{n_0}$ properly to $r_{n_0+i}$ such that $r_{n_0+i}$ will not tend to zero too fast as $i\to \infty$.   Let $M_1=M^3$. Since $|\tilde{c}_{n_0}(t)|<M-10$ if $|t|\le r_{n_0}$, inequalities \eqref{simple inequality}, \eqref{r_0 and tilde b} and \eqref{sum x_i} yield that
  $$\left|g(\tilde c_{n_0}(t))-\sum_{i=0}^{\gamma_a}x_{n+1,i}t^{2\gamma_a+i}\right|<M^3-30+3+4 <M_1-20.$$
Notice that the order of $g(\tilde c_{n_0})-\sum_{i=0}^{\gamma_a}x_{n+1,i}t^{2\gamma_a+i}$ at $t=0$  is at least $ 3\gamma_a+1$; so, again by the Maximal Value Theorem, for all $|t|\leq r_{n_0}$ we have that 
  $$\frac{\left|g(\tilde c_{n_0}(t))-\sum_{i=0}^{\gamma_a}x_{n+1,i}t^{2\gamma_a+i}\right|}{|t|^{3\gamma_a+1}}\leq \frac{M_1-20}{r_{n_0}^{3\gamma_a+1}},$$
i.e., 
  $$\left|\tilde c_{n_0+1}-\sum_{i=0}^{\gamma_a}x_{n+1,i}t^{i}\right|\leq \frac{M_1-20}{r_{n_0}^{3\gamma_a+1}}\cdot |t|^{\gamma_a+1}.$$
Let $r_{n_0+1}:=r_{n_0}^{(3\gamma_a+1)/(\gamma_a+1)}$. We have that for all $t$ with $|t|\leq r_{n_0+1}$,
   $$\left|\tilde c_{n_0+1}(t)-\sum_{i=0}^{\gamma_a}x_{n+1,i}t^{i}\right|<M_1-20.$$
Since  $r_{n_0+1}<r_{n_0}<r_0$, inequality \eqref{sum x_i} yields that if $|t|\leq  r_{n_0+1}$, then
    $$|\tilde c_{n_0+1}|<M_1-10.$$
Replacing $n_0$ and $M$ in (\ref{inductive 211}) by $n_{0}+1$ and $M_1$, we get that for all $|t|<r_{n_0+1}$ and all $i\geq 0$
    $$|\tilde c_{n_0+1+i}|< \frac{M_1^{3^i}}{r_{n_0+1}^{2\gamma_a(3^{i-1}+3^{i-2}+\cdots+3^0)}}-10 <\frac{M_1^{3^i}}{r_{n_0+1}^{2\gamma_a(3^{i-1}+3^{i-2}+\cdots+3^0)}}.$$
Consequently, 
   $$\sup_{i\geq n_0+1}\sup_{|t|<r_{n_0+1}}\frac{\log^+|\tilde c_i(t)|}{3^i}\leq \frac{\log M_1}{3^{n_0+1}}-\frac{\log r_{n_0+1}^{2\gamma_a}}{3^{n_0+1}}.$$
Because $r_{n_0+1}=r_{n_0}^{(3\gamma_a+1)/(\gamma_a+1)}$ and $M_1=M^3$, the above inequality becomes
    $$\sup_{i\geq n_0+1}\sup_{|t|<r_{n_0+1}}\frac{\log^+|\tilde c_i(t)|}{3^i}\leq \frac{\log M}{3^{n_0}}-\left(\frac{3\gamma_a+1}{3\gamma_a+3}\right)\frac{\log r_{n_0}^{2\gamma_a}}{3^{n_0}}.$$
Inductively, for all $j\geq 0$ and $r_{n_0+j}:=r_{n_0}^{(3\gamma_a+1)^j/(\gamma_a+1)^j}$
    $$\sup_{i\geq n_0+j}\sup_{|t|<r_{n_0+j}}\frac{\log^+|\tilde c_i(t)|}{3^i}\leq \frac{\log M}{3^{n_0}}-\left(\frac{3\gamma_a+1}{3\gamma_a+3}\right)^j\frac{\log r_{n_0}^{2\gamma_a}}{3^{n_0}}.$$
Since $M$ is fixed and we can pick arbitrarily large $n_0$ and $j$,  the above inequality asserts (\ref{local limit}). This concludes the proof of Lemma~\ref{continuity at bad point}. \qed

\subsection{Metrized line bundle}\label{metric line bundle} Let $L$ be a line bundle of a projective curve $C$. A {\em metric} $\|\cdot \|$ on $L$ is a collections of norms on $L(t)$, one for each $t\in C$, such that for a section $s$
   $$\|\alpha \cdot s(t)\|:=|\alpha|\|s(t)\|$$
for all constants $\alpha$. The metric $\|\cdot \|$ is said to be {\em continuous}, if $\|s(t)\|$ varies continuously as we vary $t$ locally. All metrics considered in this article are continuous. A continuous metric $\|\cdot \|$ on $L$ is said to be {\em semi-positive} if, locally $\log\|s(t)\|^{-1}$ is a subharmonic function for any non vanishing holomorphic section $s$, and the {\em curvature} of this metric is given by 
\begin{equation}\label{curvature}
c_1(L)=dd^c\log \|s(t)\|^{-1},
\end{equation}
which can be viewed as a measure on $C$. It is clear the curvature does not depend on the choice of the section $s$.

   We work on the curve where the marked critical point is not persistently prepeiodic. Then by Lemma \ref{non empty poles}, $S_0^+$ (or $S_0^-$) is non empty. For each $n\geq 1$, let $D_n^+$ be the Weil divisor in the smooth projective curve $C$ defined as 
   $$D_n^+:=\sum_{t_0\in S_0^+} \gamma_n\cdot t_0$$
where $\gamma_n$ is the order of pole of $c_n$ at $t_0$ (see \eqref{gamma n}). Similarly, we can define the divisor $D_n^-$ corresponding to the marked critical point $-a$ on $C$. By Lemma \ref{order of pole growth}, for any sufficiently large $n$ and all $i\geq 1$,  we have 
\begin{equation}\label{divisor growth}
  D^{\pm}_{n+i}=3^i\cdot D_n^\pm.
\end{equation}

Let $L_{D_n^\pm}$ be the line bundles of $C$ associated to $D_n^\pm$. From Theorem \ref{continuity thm}, when $n$ is big enough, there is a canonical metric $\|\cdot \|$ on $L_{D_n^\pm}$, which is continuous and semi-positive. Indeed, via the canonical map 
 \begin{equation}\label{canonical morph of line}
 \phi:  L_{D_n^\pm}|_{C\backslash S^{\pm}}\simeq (C\backslash S^{\pm})\times \C,
 \end{equation}
for any section $s(t)$ of $L_{D_n^\pm}$ on ${C\backslash S^{\pm}}$, the metric $\|\cdot \|$ on $s$ is given by 
\begin{equation}\label{metric definition}
\|s(t)\|:=e^{-3^n\cdot G^\pm}\cdot |\phi(s(t))|.
\end{equation}
Theorem \ref{continuity thm} asserts that when $n$ is sufficiently large, this metric can be extended to a continuous and semi-positive metric on $L_{D_n^\pm}$. If we consider the metric on $L^{\otimes i}_{D_n^\pm}$ inherited from the metric on $L_{D_n^\pm}$, then for large $n$ (using \eqref{divisor growth} and \eqref{metric definition}), we have that  
   $$L^{\otimes 3^i}_{D_n^\pm}\simeq L_{D_{n+i}^\pm}$$
is an isometry.

\section{Equidistribution of small points}
In this section, we show that the parameters, at which the marked critical point is preperiodic are equidistributed on $C$ with respect to the bifurcation measure of the marked critical point.  

\subsection{Height functions} Let $K$ be a number field and let $\Kbar$ be its algebraic closure. It is well known that $K$ is naturally equipped with a set $\Omega_K$ of pairwise inequivalent nontrivial absolute values, together with positive integers $N_v$ for each $v\in \Omega_K$ such that 
\begin{itemize}
\item for each $\alpha\in K^*$, we have $|\alpha|_v=1$ for all but finitely many places $v\in \Omega_K$. 
\item for every $\alpha\in K^*$, we have the {\em product formula}:
\begin{equation}\label{product formula}
\prod_{v\in \Omega_K}|\alpha|_v^{N_v}=1.
\end{equation}
\end{itemize}

For each $v\in \Omega_K$, let $K_v$ be the completion of $K$ at $v$, let $\Kbar_v$ be the algebraic closure of $K_v$ and let $\C_v$ denote the completion of $\Kbar_v$. The field $\C_v$ is algebraic closed; when $v$ is Archimedean, $\C_v\simeq\C$. We fix an embedding of $\Kbar$ into $\C_v$ for each $v\in \Omega_K$; hence we have a fixed extension of $|\cdot|_v$ on $K$ to $\Kbar$.  Let $f\in K[z]$ be a polynomial of degree $\geq 2$. There is a {\em canonical height} $\hhat_f$ on $\Kbar$ (see also \eqref{first canonical height}),  which is given by 
\begin{equation}\label{canonical height}
\hhat_f(x):=\frac{1}{[K(x):\Q]}\lim_{n\to \infty}\sum_{y\in \Gal(\Kbar/K)\cdot x}\sum_{v\in \Omega_K}N_v\cdot\frac{\log^+|f^n(y)|_v}{d^n}
\end{equation}
where $\Gal(\Kbar/K)\cdot x$ is Galois orbit of $x$ in $\Kbar$. As proven in \cite{Call:Silverman}, for any $x\in \Kbar$, we have that $\hhat_f(x)\geq 0$ with equality if and only if $x$ preperiodic under the iteration of $f$. Counted with multiplicity, a degree $d\geq 2$ polynomial $f$ has exactly $d-1$ critical points. As introduced by Silverman, the {\em critical height} $\hhat_{\crit}$ of a polynomial is given by the sum of the canonical heights of its $(d-1)$ critical points (see \eqref{first critical height}). Clearly, $\hhat_{\crit}(f)\geq 0$ with equality if and only if $f$ is PCF. 

\subsection{Arithmetic equidistribution theorem}\label{arith equip} Let $X$ be an irreducible projective curve defined over a number field $K$, and $L$ be an ample line bundle of $X$. Replacing the absolute value $|\cdot |$ in Section~\ref{metric line bundle} by $|\cdot|_v$, we can define metrics $\|\cdot\|_v$ on $L$ corresponding to each $v\in \Omega$. Let $X_{\C_v}^{an}$ be the analytic space associated to $X(\C_v)$. When $v$ is Archimedean, $X_{\C_v}^{an}\simeq X(\C)$.  For Archimedean $v$, a continuous metric $\|\cdot \|_v$ on a line bundle is {\em semi-positive}, if its curvature $dd^c\log \|\cdot \|_v$ is non negative. For non-Archimedean place $v$, $X_{\C_v}^{an}$ is the Berkovich space associated $X(\C_v)$, and Chambert-Loir \cite{CLoir, CLoir11} constructed an analog of curvature on $X_{\C_v}^{an}$ using
methods from Berkovich spaces.

An {\em adelic metrized line bundle} 
  $$\Lbar:=\{L, \{\|\cdot\|_v\}_{v\in \Omega_K}\}$$
over $L$ is a collection of metrics on $L$, one for each place $v\in \Omega_K$, satisfying certain continuity and coherence conditions; see \cite{Zhang:line,Zhang:metrics}. Precisely, the metric $\|\cdot\|_v$ on $L$ should be continuous for each $v\in \Omega_K$. Moreover, we require that: there exists a model $(\mathfrak{X}, \mathscr{L}, e)$ over the ring of integers of $K$ inducing the given metrics for all but finitely many places $v\in \Omega_K$. An adelic metrized line bundle is said to be {\em  semi-positive} if the metrics are semi-positive at all places; see \cite{CLoir11} for more details.

For example, we can construct an adelic metrized line bundle from a non constant morphism $g: X(K)\to \P^1(K)$ for an irreducible smooth projective curve $X$ defined over a number field $K$. Let 
  $$D:=\sum_{x\in X: ~\ord_x g<0} (-\ord_x g)\cdot x$$
be the Weil divisor corresponding to the poles of $g$. For each $v\in \Omega_K$, the metric $\|\cdot\|_v$ on $L_D|_{X\backslash D}$ is defined as
   $$\|s(x)\|_v:=e^{-\log^+|g|_v}\cdot |\phi(s(x))|_v$$
for $s$ being a section of $L_D|_{X\backslash D}$ and $\phi$ being the canonical isomorphism $L_D|_{X\backslash D}\simeq (X\backslash D)\times \C_v$. This metric can be extended to the whole $L_D$, and the extended metric is continuous and semi-positive. The metrized line bundle $\Lbar_D:=\{L_D, \{\|\cdot\|_v\}_{v\in\Omega_K}\}$ constructed this way is adelic. The simplest one is when $X=\P^1$ and $g$ is the identity map, i.e., we get an adelic metrized line bundle $\overline{\O}_{\P^1}(1)$. For general non constant morphisms $g: X\to \P^1$, it is obvious $\Lbar_D\simeq g^*\overline{\O}_{\P^1}(1)$.

For a semi-positive line bundle $\Lbar$ of $C$, there is a height $\hhat_{\Lbar}(Y)$ associated to it for each subvariety $Y$ of $X$; see \cite{Zhang:metrics}. In the case of points $x\in X(\Kbar)$, the height is given by 
\begin{equation}\label{metrized height}
\hhat_{\Lbar}(x):=\frac{1}{[K: \Q]}\sum_{y\in \Gal(\Kbar/K)\cdot x}\sum_{v\in \Omega_K}\frac{-N_v\cdot \log\|s(y)\|_v}{|\Gal(\Kbar/K)\cdot x|}
\end{equation}
where $|\Gal(\Kbar/K)\cdot x|$ is the number of points in the Galois orbit of $x$, and $s$ is any meromorphic section of $L$ with support disjoint from $\Gal(\Kbar/K)\cdot x$. A sequence of points $x_n\in X(\Kbar)$ is said to be {\em small} if $\lim_{n\to \infty} \hhat_{\Lbar}(x_n)=\hhat_{\Lbar}(X)$. For all cases considered in this article, we will see that $\hhat_{\Lbar}(X)$ is always zero.  
\begin{theorem}\label{arith equi dis thM}\cite{CLoir, Thuillier, Yuan} Suppose $X$ is a projective curve over a number field $K$ and $\Lbar$ is a semi-positive adelic metrized line bundle on $X$ with $L$ being ample. Let $\{x_n\}$ be a non-repeating sequence of points in $X(\Kbar)$ which is small. Then for any $v\in \Omega_K$, the Galois orbits of the sequence $\{x_n\}$ are equidistributed in the analytic space $X_{\C_v}^{an}$ with respect to the probability measure $\mu_v=c_1(\Lbar)_v/\deg_L(X)$. 
\end{theorem}
Here $c_1(\Lbar)_v$ is  the curvature of the metric on $L$ for each place $v\in \Omega_K$.

\subsection{Equidistribution of parameters with small heights} Analog to (\ref{escape-rate definition}), for each $v\in \Omega_K$, we define 
   $$G^\pm_v(a,b):=\lim_{n\to \infty} \frac{\log^+|f_{a,b}^n(\pm a)|_v}{3^n}.$$
There is a canonical metric $\|\cdot\|_v$ on each $L_{D^\pm_n}$ (hence canonical metrics $\|\cdot\|_v$ on $L^{\otimes i}_{D^\pm_n}$ for $i\geq 1$) for $n$ sufficient large given by (see (\ref{metric definition})) 
 \begin{equation}\label{metric for all places}
 \|s(t)\|_v:=e^{-3^n\cdot G_v^\pm}\cdot |\phi(s(t))|_v. 
 \end{equation}
More precisely, $n$ needs to be sufficiently large so that for each $t_0\in S_0^\pm$, we have that $\gamma_{m+1}^\pm=3\cdot \gamma_m^\pm$ for $m\ge n$, where $\gamma_m^\pm$ is the order of the pole at $t_0$ of $f_{a,b}(\pm a)$ (see Lemma~\ref{order of pole growth}).  For such a positive integer $n$, we have
\begin{equation}
\label{good n}
L_{D_{n+i}^\pm}=L_{D_n^\pm}^{\otimes{3^i}}\text{ for each }i\ge 1.
\end{equation}
Line by line examination of Section \ref{continuous line bundle} indicates that $\|\cdot\|_v$ on $L_{D^\pm_n}$ is continuous and semi-positive for each $v\in \Omega$. Actually,  we use only Taylor series, triangle inequality, Maximum Value Theorem and elementary algebra for the proofs of lemmas in Section \ref{continuous line bundle}, and all these work when dealing with the cases involving non-Archimedean norm $|\cdot|_v$. 

Now, let $n_0$ be a sufficiently large positive integer so that \eqref{good n} holds, and let $\Lbar_{D_n^{\pm}}:=\{L_{D_n^\pm}, \{\|\cdot\|_v\}_{v\in \Omega_K}\}$ be the metrized line bundle constructed above.

\begin{lemma}\label{adelic}
Suppose $+a$ is not persistently preperiodic under $f_{a,b}$ on $C$, then for any sufficiently large $n$, there is non constant rational function $g$ on $C$ with 
   $$3^n\cdot G^+_v=\log^+|g|_v$$
on $C$ for all but finitely many $v\in \Omega_K$. 
\end{lemma}
\proof Replacing the divisor $d^N(d-1)D$ in \cite[Lemma 6]{ingram} by $D^+_n$, we obtain the desired result from \cite[Lemma 13]{ingram}.\qed

For $g$ being the rational map as in the conclusion of Lemma~\ref{adelic}, we get that the isomorphism $L_{D_n^+}\simeq L_D:=g^*\O_{\P^1}(1)$ is an isometry for all but finitely many places $v\in \Omega_K$. Similarly, Lemma~\ref{adelic} holds for the marked critical point $-a$ and its escape-rate function $G_v^-$. 

We note that Thuillier \cite{Thuillier} defined the notion of being $W^1$-regular for a metric on a line bundle at an arbitrary place $v$ (both Archimedean and non-Archimedean); then Thuillier \cite[Theorem~4.3.7]{Thuillier} proves the equidistribution theorem for points of small height with respect to adelic, $W^1$-regular metrics on ample line bundles on a curve. Also, as proven by Thuillier (we thank Laura DeMarco and Myrto Mavraki for sharing with us Thuillier's note), every continuous subharmonic metric is $W^1$-regular, and therefore Theorem~\ref{arith equi dis thM} holds for such metrics, as it is the case for the metrics we constructed on $L_{D_{n_0}^\pm}$. We also thank Xinyi Yuan for pointing out that Theorem~\ref{arith equi dis thM} holds for the metrics constructed on $L_{D_{n_0}^\pm}$ by employing \cite[Proposition~A.8]{Yuan-Zhang}; Yuan and Zhang  prove that the limit of semipositive metrics is still semipositive, even when the
underlying line bundles in the limit process are not the same. 

Combining Lemma \ref{adelic} and Section \ref{arith equip}, one obtains the following:

\begin{cor}\label{adelic cor}
Suppose $+a$ (resp. $-a$) is not persistently preperiodic under $f_{a,b}$ on $C$, then $\Lbar_{D_{n_0}^{\pm}}$ is a semi-positive adelic metrized line bundle on $C$. 
\end{cor}

For $t\in C(\Kbar)$, we let $\hhat_\pm(t):=\hhat_{f_{a(t), b(t)}}(\pm a(t))$
 be the height of the two critical points of $f_{a,b}$. From the definition of the height corresponding to the metrized line bundle $\Lbar_{D_{n_0}^\pm}$, and any $t\in C(\Kbar)\backslash S^\pm$, one has 

\begin{eqnarray*}
\hhat_{\Lbar_{D_{n_0}^\pm}}(t)&=&\frac{1}{[K: \Q]}\sum_{y\in \Gal(\Kbar/K)\cdot t}\sum_{v\in \Omega_K}\frac{-N_v\cdot \log\|s(y)\|_v}{|\Gal(\Kbar/K)\cdot t|} \textup{ , by (\ref{metric for all places})}\\
&=& \frac{1}{[K: \Q]}\sum_{y\in \Gal(\Kbar/K)\cdot t}\sum_{v\in \Omega_K}\frac{N_v\cdot \left(3^{n_0}G_v^\pm - \log |\phi(s(y))|_v\right)}{|\Gal(\Kbar/K)\cdot t|} \textup{ , by (\ref{product formula})}\\
&=&\frac{3^{n_0}}{[K: \Q]}\sum_{y\in \Gal(\Kbar/K)\cdot t}\sum_{v\in \Omega_K}\frac{N_v\cdot G_v^\pm(a(y), b(y))}{|\Gal(\Kbar/K)\cdot t|} \\
&=&3^{n_0}\cdot \frac{\lim_{i\to \infty}\sum_{y\in \Gal(\Kbar/K)\cdot t} \sum_{v\in \Omega_K}\frac{N_v\cdot \log^+ \left|f^i_{a(y), b(y)}(\pm a(y))\right|_v}{3^i}}{[K: \Q]\cdot |\Gal(\Kbar/K)\cdot t|}\\
&=&3^{n_0}\cdot \hhat_\pm(t). 
\end{eqnarray*}
From the formula in Lemma \ref{convergence 2} (see also Theorem~\ref{continuity thm}~(ii)) for $t_0\in S_0^\pm$, using also Lemma~\ref{order of pole growth}, we conclude that 
\begin{eqnarray*}
\hhat_{\Lbar_{D_{n_0}^\pm}}(t_0)&=&\frac{1}{[K:\Q]}\sum_{v\in \Omega_K} \lim_{i\to \infty} \frac{1}{3^{i}}\log \left|u^{\gamma_{n_0+i}}f^{n_0+i}_{a,b}(\pm a)(t_0)\right|_v\\
&=&\frac{1}{[K:\Q]}\sum_{v\in \Omega_K} \lim_{i\to \infty}\log |u^{\gamma_{n_0}}f^{n_0}_{a,b}(\pm a)(t_0)|_v\\
&=&\frac{1}{[K:\Q]}\sum_{v\in \Omega_K} \log |u^{\gamma_{n_0}}f^{n_0}_{a,b}(\pm a)(t_0)|_v\\
&=& 0.
\end{eqnarray*}
The first equality above follows since
\begin{equation}
\label{equality at t_0}
\left(u^{\gamma_{n+1}}c_{n+1}^\pm\right)(t_0) = \left((u^{\gamma_n}c_n^\pm)(t_0)\right)^3 
\end{equation}
for each $n\ge n_0$, where (as always) $c_n^\pm:=f_{a,b}^n(\pm a)$. Indeed, since 
$$c_{n+1}=c_n^3-3a^2c_n+b,$$
and $c_n$ has a pole at $t_0$ of order $\gamma_n$, which is larger than $\max\{-\ord_{t_0}(a), -\ord_{t_0}(b)/3\}$ (see Lemma~\ref{order of pole growth}), then we get \eqref{equality at t_0}. 

Also for $t_0\in S^\pm\backslash S^\pm_0$, from (\ref{local limit}) which yields $G_v^\pm(a(t_0), b(t_0))=0$ (see also Theorem~\ref{continuity thm}~(i)) for each $v\in \Omega_K$; therefore $\hhat_{\Lbar_{D_{n_0}^\pm}}(t_0)=0$.

\begin{theorem}\label{equidistribution of parameters}
Let $C\subset \C^2$ be an irreducible curve defined over some number field $K$ such that $\pm a$ is not persistently preperiodic for $f_{a,b}$ on $C$, and let $n_0$ be a sufficiently large integer (as in the conclusion of Lemma~\ref{order of pole growth}). Then for any non-repeating sequence of points $t_n\in C(\Kbar)$ with $\lim_{n\to \infty} \hhat_{\Lbar_{D_{n_0}^\pm}}(t_n)=0$, the Galois orbits of $\{t_n\}_{n\geq 1}$ equidistribute with respect to the probability measure $\mu_v^\pm=c_1(\Lbar_{D_{n_0}^\pm})_v/\deg (D^\pm_{n_0})$ on $C_{\C_v}^{an}$ for each $v\in \Omega_K$. 
\end{theorem}
\proof To prove this theorem, by Theorem \ref{arith equi dis thM} and Corollary \ref{adelic cor}, it suffices to show that $\hhat_{\Lbar_{D_{n_0}^\pm}}(C)=0$. We know that $\hhat_{\Lbar_{D_{n_0}^\pm}}(t)=3^{n_0}\cdot \hhat_\pm(t)\geq 0$ for $t\in C\setminus S^\pm$ with equality if and only $\pm a(t)$ is preperiodic under $f_{a(t), b(t)}$, we know that $\hhat_{\Lbar_{D_{n_0}^\pm}}$ is non-negative on $C$ (note also that $\hhat_{\Lbar_{D_{n_0}}}(t)=0$ if $t\in S^\pm$). Also, there are infinitely many $t\in C(\Kbar)$ with $\hhat_{\Lbar_{D_{n_0}^\pm}}(t)=0$ (see Proposition \ref{infinite preperiodic points}). Hence $\hhat_{\Lbar_{D_{n_0}^\pm}}(C)=0$ by \cite[Theorem 1.10]{Zhang:metrics}. \qed

\begin{remark}\label{remark 11}
If $v$ is Archimedean, then $C_{\C_v}^{an}\simeq C(\C)$ and $\mu_v^\pm$ is the normalization (total mass $1$ on $C$) of the bifurcation measure $\mu_\pm$ introduced in (\ref{bifurcation measures}). If $C$ is defined over $\Qbar$, then a point $t\in C(\Qbar)\setminus S^\pm$ has height zero for the adelic metrized line bundle $\Lbar_{D_{n_0}^\pm}$ if and only if the critical point $\pm a(t)$ is preperiodic under $f_{a(t), b(t)}$. Hence the set of parameters on $C$ for which the marked critical point $\pm a$ preperiodic under $f_{a,b}$  equidistribute with respect to the bifurcation measure $\mu_v^\pm$. 
\end{remark}

\begin{cor}\label{proportional escape cor}
 Let $C\subset \C^2$ be an irreducible curve defined over a number field $K$, satisfying 
\begin{itemize}
\item both $+a$ and $-a$ are not persistently preperiodic under $f_{a,b}$ on $C$. 
\item there is a sequence of non-repeating points $t_n\in C(\Kbar)$ with
   $$\lim_{n\to \infty}\hhat_{\crit}(f_{a(t_n),b(t_n)})=0.$$
\end{itemize}
Then for any sufficiently large $n_0$,  $\deg ({D^-_{n_0}})\cdot G_v^+=\deg ({D^+_{n_0}})\cdot G_v^-$ on $C$ for all $v\in \Omega_K$. 
\end{cor}
\proof As $\hhat_{\crit}(f_{a(t),b(t)})=\hhat_+(t)+\hhat_-(t)=\left( \hhat_{\Lbar_{D_{n_0}^+}}(t)+ \hhat_{\Lbar_{D_{n_0}^-}}(t)\right)/3^{n_0},$ we have 
   $$\lim_{n\to \infty} \hhat_{\Lbar_{D_{n_0}^-}}(t_n)=\lim_{n\to \infty} \hhat_{\Lbar_{D_{n_0}^+}}(t_n)=0.$$
Consequently, the Galois orbits of $\{t_n\}$ equidistribute on $C_{\C_v}^{an}$ with respect to the measures $\mu_v^\pm$ appearing in Theorem \ref{equidistribution of parameters}; hence we have 
   $$\mu_v^+=\mu_v^-.$$
When $v$ is Archimedean, $C_{\C_v}^{an}\simeq C(\C)$ and from (\ref{bifurcation measures}, \ref{curvature}, \ref{metric definition}), one has 
  $$3^{n_0}\cdot \mu_\pm=\deg (D_{n_0}^\pm)\cdot \mu_v^\pm.$$
Combining the above two formulas with Proposition \ref{proportional escape-rate}, then for each Archimedean $v\in \Omega_K$ we get the following equality on $C$:
  $$\deg (D_{n_0}^-)\cdot G_v^+=\deg (D_{n_0}^+)\cdot G_v^-.$$
By looking at the growth of $G^\pm_v(a(t), b(t))$ for $t\to t_0\in S_0^\pm$ (see Theorem \ref{continuity thm}),  the above equation indicates $\deg (D_{n_0}^-)\cdot D_{n_0}^+=\deg (D_{n_0}^+)\cdot D_{n_0}^-.$ So we have $L_{D_{n_0}^+}^{\otimes \deg (D_{n_0}^-)}\simeq L_{D_{n_0}^-}^{\otimes \deg (D_{n_0}^+)}$, and the two canonical metrics on this line bundle induce the same curvature on $C$ for each $v\in \Omega_K$ because $\mu^+_v=\mu^-_v$. From \cite[Theorem~3.3]{Yuan12}, we know that if two continuous semi-positive metrics on a line bundle have the same curvature, then these two metrics are proportional to each other. So for each $v\in \Omega_K$, we have 
   $$\deg (D_{n_0}^-)\cdot G_v^+=\deg (D_{n_0}^+)\cdot G_v^-+M(v)$$
on $C$ with $M(v)$ being a constant  depending only on $v$. Because for all $t\in C$, we have
   $$\hhat_{\crit}(f_{a(t), b(t)})\geq \hhat_\pm(t)\geq G_v^\pm(a(t), b(t))\geq 0,$$
we have  $\lim_{n\to \infty}G_v^+(a(t_n), b(t_n))=\lim_{n\to \infty}G_v^-(a(t_n), b(t_n))=0$, i.e., $M(v)=0$ for all $v\in \Omega$. \qed


\section{An algebraic relation between critical points} 
In this section, under certain conditions, we prove an algebraic relation of the iterated critical points on a curve $C$; see Theorem \ref{algebraic relation}. 

Let $K$ be a number field and $C\subset \C^2$ is an irreducible curve defined over $K$.  Denote 
   $$c_n^\pm:=f^n_{a,b}(\pm a)\in K(C)$$
\begin{theorem}\label{algebraic relation}
 Suppose the following two conditions are satisfied:
\begin{itemize}
\item both $+a$ and $-a$ are not persistently preperiodic under $f_{a,b}$ on $C$. 
\item there is a sequence of non-repeating points $t_n\in C(\Kbar)$ with
   $$\lim_{n\to \infty}\hhat_{\crit}(f_{a(t_n),b(t_n)})=0.$$
\end{itemize}
Then there exist non constant polynomials $P_{\pm}(z)\in K(C)[z]$ with
   $$P_+(c_n^+)=P_-(c_{n}^-)\in K(C)$$
for all sufficiently large $n$. 
\end{theorem}

\subsection{B$\ddot{\textup{o}}$ttcher coordinate} For each $f_{a,b}(z)=z^3-3a^2z+b$ there is a unique uniformized B$\ddot{\textup{o}}$ttcher coordinate $\Phi_{a,b}(z)$, i.e., a univalent, analytic function defined on a neighbourhood of infinity, which is uniquely determined by the conditions 
   $$\Phi_{a, b}(f_{a,b}(z))=\Phi_{a,b}(z)^3, \textup{ and } \Phi_{a,b}(z)=z+o(1).$$
Moreover, for the escape-rate function of the polynomial $f_{a,b}$
   $$G_{a,b}(z):=\lim_{n\to \infty} \frac{\log^+|f^n_{a,b}(z)|}{3^n},$$
we have 
   $$\log |\Phi_{a,b}(z)|=G_{a,b}(z)$$
 for any $z\in \C$ with $G_{a, b}(z)>\max \{G_{a,b}(+a), G_{a,b}(-a)\}$. Let 
   $$\Phi_{a,b}(z)=:z+\frac{\alpha_1}{z}+\frac{\alpha_2}{z^2}+\frac{\alpha_3}{z^3}+\cdots$$
 From the relation $\Phi_{a, b}(f_{a,b}(z))=\Phi_{a,b}(z)^3$, it is easy to see that $\alpha_i$ is a polynomial of $a$ and $b$ for all $i\geq 1$ by induction. Moreover, let the $k$-th power of $\Phi_{a,b}(z)$ be written as:
\begin{equation}
\label{definition P_m}
\Phi_{a,b}(z)^k= P_k(z)+\frac{\alpha_{k,1}}{z}+\frac{\alpha_{k, 2}}{z^2}+\frac{\alpha_{k, 3}}{z^3}+\cdots
\end{equation}
where  $P_k(z)\in \C[a,b][z]$ and $\alpha_{k, i}$ are polynomial functions in $a$ and $b$. 

The following result is implicitly used in the proof of \cite[Lemma~5.5]{Matt-Laura-2} and the ingredients for its proof are all contained in \cite{Matt-Laura-2}, as it was kindly pointed out to us by Laura DeMarco, whom we thank warmly for her help.
 
\begin{lemma}\label{high order lemma}
We work under the above notation for the curve $C$ and the corresponding $\Phi_{a(t), b(t)}$ for the specialization of the B$\ddot{\textup{o}}$ttcher's coordinate along the curve $C$; we also let $f_t(z):=z^3 -3a(t)^2z + b(t)$ for $t\in C(\C)$. Let $t_0\in S_0^\pm$ and  let $k\in\N$. Then for each $c\in \C(C)$ with the property that $-\ord_{t_0}f_t^n(c(t))\to \infty$,   we have that $$\ord_{t_0}\left(\Phi_{a(t),b(t)}(f_t^n(c(t)))^k - P_k(f_t^n(c(t)))\right)\to \infty,$$
as $n\to\infty$.
\end{lemma}

\begin{proof}
First, we know that each $\alpha_{k,i}$ is a polynomial in $a$ and $b$ and so, specializing $a$ and $b$ along the curve $C$ we obtain that each $\alpha_{k,i}$ becomes a function $\alpha_{k,i}(t)\in \C(C)$. We note that the order of the pole of $\alpha_{k, i}(t)$ at $t_0$ grows at most polynomially fast as $i\to \infty$ (for $k$ fixed) (see \cite[Lemma~5.4]{Matt-Laura-2}). 

Secondly, arguing as in the proof of Lemma~\ref{continuity at bad point}, we choose a suitable analytic parametrization of a neighborhood of $t_0$ and a suitable uniformizer $u$ at $t_0$ such that we may assume $t_0=0$ and $u=t$. Thus (from the definition of the B$\ddot{\textup{o}}$ttcher's coordinate) we obtain that the power series 
$$F(t, z) := \sum_{i\ge 1} \frac{\alpha_{k,i}(t)}{z^i},$$ 
converges on open sets of $\C\times \C$ of the form 
\begin{equation}
\label{open sets form}
\{(t,z)\colon |t| < s\text{ and }|z|>R_t\},
\end{equation} 
for any  $s<s_0$,  where $s_0$ is a given small positive real number, and $|z| > R_t$, where $R_t$ is larger than the radius of convergence for the B$\ddot{\textup{o}}$ttcher's coordinate of the polynomial $f_t(z)$; for example, one may take (see the proof of \cite[Lemma~5.1]{Matt-Laura-2})
\begin{equation}
\label{R_t}
R_t=O\left(|t|^{2\min\{\ord_0(a), \ord_0(b)/3\}}\right).
\end{equation}
Moreover, we let 
$$R_s := \max\{R_t : |t| = s\}$$
for any $s<s_0$; thus $R_s<+\infty$. Also, since $t_0\in S_0^\pm$, at the expense of replacing $s_0$ by a smaller positive real number, we know that the function $s\mapsto R_s$ is decreasing on $(0,s_0)$ (see \eqref{R_t}). Next we consider the function 
\begin{equation}
\label{defining U ell}
U^\ell(t) := \sum_{i\ge 1} \frac{\alpha_{k,i}(t)}{f_t^\ell(c(t))^i},
\end{equation}
for $\ell\in\N$, 
where $c\in \C(C)$ is a function such that $\lim_{n\to\infty}-\ord_{0}f_t^n(c(t))=\infty$. Our goal is to conclude that 
\begin{equation}
\label{conclusion growth U_l}
\lim_{\ell\to\infty}\ord_{0}U^\ell(t)=+\infty.
\end{equation}

Now, since $F(t,z)$ is analytic on open sets of the form \eqref{open sets form}, then the series defining $F(t, z)$ converges uniformly on sets of the form  
$$\{t\in C(\C) : s_1 < |t| < s_2\} \times \{z: |z| > R_{s_1}\},$$
where $0<s_1<s_2<s_0$.

Fix now some $s_1<s_2$ in the interval $(0,s_0)$. Then let $\ell\in\N$  with the property that if $|t|<s_2$, we have that $|f_t^\ell(c(t))| > R_{s_1}$. Then the series $U^\ell(t)$ converges uniformly on 
$$\{t\in C(\C)\colon  s_1 < |t| < s_2\}.$$ 
In particular, this means that the difference between $U^\ell(t)$ and the partial sum $U^\ell_n(t)$ which is the sum of the first $n$ terms in the series \eqref{defining U ell} is uniformly bounded on the above annuli.

However, because both $U^\ell(t)$ and also $U^\ell_n(t)$ are analytic functions when $|t| <s_2$, then their difference is bounded by their difference on the closed set $|t| = s_2$ (according to the Maximum Value Theorem), which is in turn uniformly bounded (as we vary $n$) because of the above uniform convergence on the annuli. So, $\{U^\ell_n\}$ converges uniformly to $U^\ell$ when $|t| < s_2$. Since each $U^\ell_n$ vanishes at $0$ of high order (note that each $U^\ell_n$ is simply a finite sum and each term vanishes at $t_0$ of order depending on the order of
vanishing for $f_t^\ell(c(t))$), then also $U^\ell(t)$ vanishes at $0$ of high order, thus proving \eqref{conclusion growth U_l}.
\end{proof}

\subsection{Proof of Theorem \ref{algebraic relation}}

From the hypotheses of Theorem \ref{algebraic relation}, by Corollary \ref{proportional escape cor}, it is clear that for each large $n_0$, 
   $$\deg(D_{n_0}^-)\cdot D_{n_0}^+=\deg(D_{n_0}^+)\cdot D_{n_0}^-.$$
Let $t_0\in S^\pm_0$; then for $t\in C$ near $t_0$, we know that for large n, both $c_n^\pm(t)$ are in the univalent domain of $\Phi_{a(t), b(t)}(z)$ (this is proven in \cite[Lemma~5.1]{Matt-Laura-2}). Since 
  $$|\Phi_{a(t), b(t)}(c_n^\pm(t))|=3^n\cdot G_{a(t),b(t)}(\pm a(t))=3^n\cdot G^\pm(a(t), b(t))$$
by Corollary \ref{proportional escape cor}, for all $t$ close to $t_0$, one has 
  $$\deg(D_{n_0}^-)\cdot |\Phi_{a(t), b(t)}(c_n^+(t))|=\deg(D_{n_0}^+)\cdot |\Phi_{a(t), b(t)}(c_n^-(t))|,$$
i.e., there is $\zeta_{t_0,n}\in \C$ of absolute value $1$ such that   
   $$\Phi^{\deg(D_{n_0}^-)}_{a(t), b(t)}(c_n^+(t))=\zeta_{t_0,n} \cdot \Phi^{\deg(D_{n_0}^+)}_{a(t), b(t)}(c_n^-(t)).$$ 
\begin{lemma}
The number $\zeta_{t_0,n}$ is a root of unity. 
\end{lemma}
\proof 
Since  $\Phi_{a(t), b(t)}(c_n^\pm(t))=c_n^\pm(t)+o(1)$, then by Lemma \ref{high order lemma} we see that  
\begin{equation}
\label{formula zeta t_0}
\zeta_{t_0,n}=\lim_{t\to t_0} \frac{(c^-_n(t))^{\deg(D_{n_0}^+)}}{(c^+_n(t))^{\deg(D_{n_0}^-)}}.
\end{equation}
In particular, this yields that $\zeta_{t_0,n}\in K$ (because $a$ and $b$ are rational functions on the curve $C$ which is defined over $K$).  
Now we want to show that for any non-Archimedean place $v\in \Omega_K$, we also have $|\zeta_{t_0,n}|_v=1$. For large $n$, we know that for $t$ close to $t_0$ in the topology determined by the non-Archimedean place $v$, 
    $$|f_{a,b}(c_n^\pm (t))|_v=|c_n^\pm(t)|_v^3>>|2a^2(t)|_v, |b(t)|_v$$
and inductively $$|f_{a,b}^i(c_n^\pm (t))|_v=|c_n^\pm(t)|_v^{3^i}.$$
Hence from the definition of $G^\pm_v$, we have $G_v^\pm(a(t), b(t))=\frac{\log |c_n^\pm(t)|_v}{3^n}$. Again, by Corollary \ref{proportional escape cor}, for $t\in C$ close to $t_0$, we have
\begin{equation}
\label{formula Green v}
\deg(D_{n_0}^-) \cdot \log |c_n^+(t)|_v= \deg(D_{n_0}^+) \cdot \log |c_n^-(t)|_v.
\end{equation}
Then equalities \eqref{formula zeta t_0} and \eqref{formula Green v} yield  $|\zeta_{t_0,n}|_v=1$. As $|\zeta_{t_0,n}|_v=1$ for all $v\in \Omega_K$, we conclude that $\zeta_{t_0,n}$ a root of unity. 
\qed

Since $\Phi_{a, b} (c_{n+i}^\pm(t))=\Phi_{a, b} (f^i_{a,b}(c_{n}^\pm(t)))=\Phi_{a,b}^{3^i}(c_n(t))$, 
if we increase $n$ to $n+i$, then we get that $\zeta_{t_0,n+i}=\zeta_{t_0,n}^{3^i}$. 

We pick a large $k$, such that $\zeta_{t_0,n}^k=1$ for all $t_0\in S_0^\pm$ (and all large $n$). Also we let 
   $$P_+(z)=P_{k\cdot \deg(D_{n_0}^-)}(z)\text{ and }P_-(z)=P_{k\cdot \deg(D_{n_0}^+)}(z),$$
where the polynomials $P_m(z)$ are defined as in \eqref{definition P_m}. 
Then Lemma \ref{high order lemma} coupled with  the fact 
    $$\Phi^{k\cdot \deg(D_{n_0}^-)}_{a(t), b(t)}(c_n^+(t))=\zeta_{t_0,n}^k\cdot \Phi^{k\cdot \deg(D_{n_0}^+)}_{a(t), b(t)}(c_n^-(t))=\Phi^{k\cdot \deg(D_{n_0}^+)}_{a(t), b(t)}(c_n^-(t)), $$ 
yield that for each $t_0\in S_0^+=S_0^-$, we have 
\begin{equation}
\label{zeros not poles}
\lim_{n\to \infty} \ord_{t_0}\left( P_+(c_n^+(t))-P_-(c_{n}^-(t))\right)=\infty.
\end{equation}
Moreover, as $P_\pm(z)\in \C[a,b][z]$  and also $c_n^\pm\in \C[a,b]$, then $P_\pm(c_n^\pm(t))$ may only have poles only at the points contained in $S^\pm$. However, \eqref{zeros not poles} yields that $P_\pm(c_n^\pm(t))$ may only have poles (of bounded order) at the finitely many points contained in $S^\pm\backslash S_0^\pm$ for all $n$. But then for large $n$, \eqref{zeros not poles} yields that  
   $$P_+(c_n^+)-P_-(c_{n}^-)=0\in K(C).$$\qed

\section{Proof of the main theorem}
In this section we finish the proof of the main theorem stated in the introduction. 

\begin{proof}[Proof of Theorem~\ref{precise version}.] Notice that PCF points $(a,b)\in \C^2$ are the intersections of the zero loci of $f^{n_1}_{a,b}(+a)-f^{m_1}_{a,b}(+a)=0$ and $f^{n_2}_{a,b}(-a)-f^{m_2}_{a,b}(-a)=0$ for some $n_i>m_i\geq 0$ with $i=1,2$. Because there are only countable many PCF points in $\M_3$, all the parameters $(a,b)$ with $f_{a,b}$ being PCF are in $\Qbar^2\subset \C^2$. If an irreducible curve $C\subset \C^2$ contains a Zariski-dense set of points in $\Qbar^2$, then $C$ must be a curve defined over $\Qbar$. Since $\hhat_{\crit}(f_{a, b})\geq 0$ with equality if and only if $f_{a,b}$ is PCF, we conclude that statement~(1) implies statement~(3) in the conclusion of Theorem~\ref{precise version}. 

Suppose now that statement~(2) holds. If one of the marked critical point is persistently preperiodic on $C$, then by Proposition \ref{infinite preperiodic points}, there are infinitely many $(a,b)\in C$ with $f_{a,b}$ being PCF. If $f^n_{a,b}(+a)=f^m_{a,b}(-a)$ on $C$ for some $n,m\geq 0$, then for any $(a,b)\in C$, $+a$ is preperiodic if and only if $-a$ is preperiodic under $f_{a,b}$. Moreover, if $b=0$ on $C$, then $f^n_{a,0}(-a)=-f^{n}_{a,0}(+a)$ for all $n$ on $C$, i.e., $+a$ has finite forward orbit under $f_{a,0}$ if and only if so does $-a$. Hence by Proposition \ref{infinite preperiodic points}, in all cases there are infinitely many $(a,b)\in C$ with $f_{a,b}$ being PCF. Therefore, statement~(2) implies statement~(1) in the conclusion of Theorem~\ref{precise version}. 

The only implication left to prove is to show that statement~(3) implies  statement~(2) in the conclusion of Theorem~\ref{precise version}. So, suppose now that statement~(3) holds. Let $K$ be a number field such that $C$ is defined over $K$. If one of the marked critical points $\pm a$ is persistently preperiodic, then the second statement holds. We assume neither of the critical points $\pm a$ is persistently preperiodic under $f_{a,b}$ on $C$.  From Theorem \ref{algebraic relation}, there exist polynomials $P_\pm(z) \in K(C)[z]$ such that for all $n\ge n_0$ (for some large positive integer $n_0$), we have 
   that $P_+(c_n^+)=P_-(c_{n}^-)$ as functions in $K(C)$. 
Consider the plane curve given by the equation $$\{(x,y)\colon P_+(x)-P_-(y)=0\}.$$ 
Since $\{(c_n^+, c^-_{n})\}_{n\geq n_0}\subset K(C)\times K(C)$ 
is an infinite set lying on $P_+(x)-P_-(y)=0$ and invariant by the coordinatewise action of $f_{a,b}$ on $\mathbb{A}^2$,  
we can find an irreducible component of $P_+(x)-P_-(y)=0$ containing infinitely many points in $\{(c_n^+, c^-_{n})\}_{n\geq n_0}$ and periodic under $(f_{a,b}, f_{a,b})$. By \cite[Theorem 6.24]{M-S-1}, such an irreducible component of $P_+(x)-P_-(y)=0$ is a graph given by
   $$x=g(y)\textup{ or } y=g(x)$$
for some $g(z)\in K(C)[z]$ which commutes with $f^\ell_{a,b}(z)\in K(C)[z]$ for some $\ell>0$. Without loss of generality, we assume the curve is given $x=g(y)$ and then $c_{n_1}^+=g(c^-_{n_1})$ for some large positive integer $n_1$. At all but finitely many points $t\in C$ (as long as they are not poles for the coefficients of   $g(z)\in K(C)[z]$), specializing $g$ at each such point yields a polynomial $g_t(z)\in \C[z]$. For $t\in C$ such that $a(t), b(t)\neq 0,\infty$, there is no symmetry for the Julia set $J_{f_{a(t), b(t)}}$. Since  the polynomial $g_t(z)$ commutes with $f_{a(t), b(t)}^\ell(z)$, we conclude that $g_t(z)$ must be an iterate of $f_{a(t), b(t)}$ (see \cite{SS} and also \cite[Proposition 2.3]{Kh}). Obviously, for all but finitely many points $t\in C$, we also have that $\deg(g_t)=\deg(g)$. Therefore there exists a nonnegative integer $r$ such that $g=f_{a,b}^r$. So if $C\subset \C^2$ is not the line given by $a=0$ or $b=0$, then we can find integers $n, m\geq 0$ with $f_{a,b}^n(+a)=f_{a,b}^m(-a)$ on $C$. For the case $a=0$, we see $f_{0, b}^0(+0)=f_{0, b}^0(-0)=0$. This concludes the proof of Theorem~\ref{precise version}. 
\end{proof}

\bigskip

\def\cprime{$'$}

\end{document}